\documentclass[11pt, a4paper]{amsart}

\usepackage{amsfonts,amsmath,amssymb, amscd}
\usepackage{pinlabel}

\newtheorem{theorem}{Theorem}[section]

\newtheorem{definition}[theorem]{Definition}
\newtheorem{prop}[theorem]{Proposition}
\newtheorem{coro}[theorem]{Corollary}

\theoremstyle{definition}
\newtheorem{rem}[theorem]{Remark}

\newtheorem{exa}[theorem]{Example}

\numberwithin{equation}{section}

\newcommand\eps{\varepsilon}
\newcommand\ZZ{\mathbb Z}
\newcommand\QQ{\mathbb{Q}}
\newcommand\CC{\mathbb{C}}
\newcommand\FF{\mathbb{F}}

\newcommand\lla{\langle\!\langle}
\newcommand\rra{\rangle\!\rangle}

\DeclareMathOperator{\Tr}{Tr}

\title[Algebraicity of zeta functions]
{Algebraicity of the zeta function\\
associated to a matrix over\\ a free group algebra}

\author[C. Kassel]{Christian Kassel}
\address{Christian Kassel: 
Institut de Recherche Math\'e\-ma\-tique Avanc\'ee,
CNRS \& Universit\'e de Strasbourg,
7 rue Ren\'{e} Descartes, 67084 Strasbourg, France}
\email{kassel@math.unistra.fr}
\urladdr{www-irma.u-strasbg.fr/\raise-2pt\hbox{\~{}}kassel/}

\author[C. Reutenauer]{Christophe Reutenauer}
\address{Christophe Reutenauer:
Math\'ematiques, Universit\'e du Qu\'ebec \`a Montr\'eal,
Montr\'eal, CP 8888, succ.\ Centre Ville, Canada H3C 3P8}
\email{reutenauer.christophe@uqam.ca}
\urladdr{www.lacim.uqam.ca/\raise-2pt\hbox{\~{}}christo/}



\begin{document}

\begin{abstract}
Following and generalizing a construction by Kontsevich,
we associate a zeta function to any matrix with entries 
in a ring of noncommutative Laurent polynomials with integer coefficients.
We show that such a zeta function is an algebraic function.
\end{abstract}

\maketitle

\section{Introduction}\label{sec-main}

Fix a commutative ring~$K$. 
Let $F$ be a free group on a finite number of generators $X_1, \ldots, X_n$ and 
\begin{equation*}
K F = K\langle X_1, X_1^{-1}, \ldots, X_n, X_n^{-1}\rangle
\end{equation*}
be the corresponding group algebra: 
equivalently, it is the algebra of noncommutative Laurent polynomials with coefficients in~$K$.
Any element $a \in K F$ can be uniquely written as a finite sum of the form
\begin{equation*}
a = \sum_{g \in F} \, (a,g) \, g \, ,
\end{equation*}
where $(a,g) \in K$.

Let $M$ be a $d\times d$-matrix with coefficients in~$K F$.
For any $n\geq 1$ we may consider the $n$-th power~$M^n$ of~$M$ and its trace $\Tr(M^n)$, 
which is an element of~$K F$.
We define the integer~$a_n(M)$ as the coefficient of~$1$ in the trace of~$M^n$:
\begin{equation}\label{an}
a_n(M) = \left( \Tr(M^n), 1 \right) \, .
\end{equation}
Let $g_M$ and $P_M$ be the formal power series
\begin{equation}\label{PM}
g_M = \sum_{n\geq1} \, a_n(M) \, t^n
\quad\text{and}\quad
P_M = \exp\left( \sum_{n\geq1} \, a_n(M) \, \frac{t^n}{n} \right).
\end{equation}
They are related by
\begin{equation*}
g_M = t \, \frac{d \log(P_M)}{dt} \, .
\end{equation*}

We call $P_M$ the \emph{zeta function} of the matrix~$M$
by analogy with the zeta function of a noncommutative formal power series
(see Section\,2.1); the two concepts will be related in Proposition\,4.1. 

The motivation for the definition of~$P_M$ comes from the well-known identity expressing the inverse of the 
reciprocal polynomial of the characteristic polynomial of a matrix~$M$ with entries in a commutative ring
\begin{equation*}
\frac{1}{\det(1 - tM)} = \exp\left( \sum_{n\geq1} \, \Tr(M^n) \, \frac{t^n}{n} \right) .
\end{equation*}

Note that for any scalar $\lambda \in K$, the corresponding series for the matrix~$\lambda M$ become
\begin{equation}\label{rescaling}
g_{\lambda M}(t) = g_M(\lambda t)
\quad\text{and}\quad
P_{\lambda M}(t) = P_M(\lambda t) \, .
\end{equation}

Our main result is the following; it was inspired by Theorem~1 of~\cite{Ko}.

\begin{theorem}\label{main}
For each matrix $M \in M_d(K F)$ where $K= \QQ$ is the ring of rational numbers, the formal power series $P_M$ is algebraic.
\end{theorem}

The special case $d=1$ is due to Kontsevich~\cite{Ko}. 
A combinatorial proof in the case $d=1$ and $F$ is a free group on one generator
appears in~\cite{RR}.
The algebraicity of the series~$g_M$, which is the logarithmic derivative of~$P_M$, was already known
(see\,\cite{Sa} and also\,\cite{GB}).

Observe that by the rescaling equalities\,\eqref{rescaling} it suffices to prove the theorem when $K= \ZZ$ is the ring of integers.

It is crucial for the veracity of Theorem~\ref{main}
that the variables do not commute: 
for instance, if $a = x+y+x^{-1}+y^{-1} \in \ZZ[x, x^{-1}, y, y^{-1}]$, 
where $x$ and $y$ are commuting variables,  then 
$\exp(\sum_{n\geq 1}\, (a^n, 1) \, t^n/n)$ is a formal power series with integer coefficients, 
but not an algebraic function
(this follows from Example~3 in\,\cite[Sect.\,1]{Bm}).

The paper is organized as follows.
In Section\,\ref{sec-cyclic} we define the zeta function~$\zeta_S$ of a 
noncommutative formal power series~$S$ and show that it can be
expanded as an infinite product 
under a cyclicity condition that is satisfied by the characterististic series of cyclic languages.

In Section\,\ref{sec-alg} we recall the notion of an algebraic noncommutative formal power series
and some of their properties.

In Section~\ref{proof-main} we reformulate the zeta function of a matrix as 
the zeta function of a noncommutative formal power series before giving
the proof of Theorem~\ref{main}; the latter follows the steps sketched in~\cite{Ko} 
and relies on the results of the previous sections as well as on an algebraicity result
by Andr\'e\,\cite{An2} elaborating on an idea of D.\ and G.~Chudnovsky.

We concentrate on two specific matrices in Section~\ref{sec-exa}.
We give a closed formula for the zeta function of the first matrix; 
its nonzero coefficients count the planar rooted bicubic maps
as well as Chapoton's ``new intervals'' in a Tamari lattice (see~\cite{Cp, Tu}).

\section{Cyclic formal power series}\label{sec-cyclic}

\subsection{General definitions}\label{general}

As usual, if $A$ is a set, we denote by $A^*$ the free monoid on~$A$: it consists of all words on the
alphabet~$A$, including the empty word~$1$.
Let $A^+ = A - \{ 1\}$.

Recall that $w\in A^+$ is \emph{primitive} if it cannot be written as $u^r$ for any integer $r\geq 2$ and any $u\in A^+$.
Two elements $w,w' \in A^+$ are \emph{conjugate} if $w = uv$ and $w'= vu$ for some $u,v \in A^*$.

Given a set~$A$ and a commutative ring~$K$, 
let $K\lla A \rra$ be the algebra of noncommutative formal power series on the alphabet~$A$. 
For any element $S \in K\lla A \rra$ and any $w\in A^*$, we define the coefficient $(S,w) \in K$ by
\begin{equation*}
S = \sum_{w\in A^*} \, (S,w) \, wÊ\, .
\end{equation*}

As an example of such noncommutative formal power series, take the characteristic series 
$\sum_{w \in L} \, w$
of a language $L \subseteq A^*$.
In the sequel we shall identify a language with its characteristic series.

The \emph{generating series} $g_S$ of an element $S\in K\lla A \rra$ is the image of~$S$ under the algebra map
$\eps : K\lla A \rra \to K[[t]]$ sending each $a \in A$ to the variable~$t$. We have
\begin{equation}\label{def-g}
g_S - (S,1) = \sum_{w\in A^+} \, (S,w) \, t^{|w|}
= \sum_{n\geq 1} \left( \sum_{|w|=n} \, (S,w) \right)  t^n \, ,
\end{equation}
where $|w|$ is the length of~$w$.

The \emph{zeta function} $\zeta_S$ of $S\in K\lla A \rra$ is defined by
\begin{equation}\label{def-z}
\zeta_S 
= \exp\left( \sum_{w\in A^+} \, (S,w) \, \frac{t^{|w|}}{|w|} \right)
= \exp\left( \sum_{n\geq 1} \left(\sum_{|w| = n} \, (S,w) \right) \frac{t^n}{n} \right)  .
\end{equation}

The formal power series $g_S$ and $\zeta_S$ are related by
\begin{equation}\label{dlog}
t \, \frac{d \log(\zeta_S)}{dt} = t \, \frac{\zeta'_S}{\zeta_S} = g_S - (S,1) \, ,
\end{equation}
where $\zeta'_S$ is the derivative of~$\zeta_S$ with respect to the variable~$t$.

\subsection{Cyclicity}

\begin{definition}
An element $S \in K\lla A \rra$ is cyclic if
\begin{itemize}
\item[(i)] 
$\forall u,v \in A^*, \;\; (S,uv) = (S,vu)$ and

\item[(ii)]
$\forall w \in A^+, \forall r \geq 2, \;\; (S,w^r) = (S,w)^r \, .$
\end{itemize}
\end{definition}

Cyclic languages provide examples of cyclic formal power series.
Recall from~\cite[Sect.~2]{BR1} that a language $L \subseteq A^*$ is \emph{cyclic} if 
\begin{itemize}
\item[(1)] 
$\forall u,v \in A^*, \;\; uv \in L \Longleftrightarrow vu \in L\, ,$

\item[(2)]
$\forall w \in A^+, \forall r \geq 2, \;\; w^r \in L \Longleftrightarrow w \in L \, .$
\end{itemize}
The characteristic series of a cyclic language is a cyclic formal power series in the above sense.

Let $L$ be any set of representatives of conjugacy classes of primitive elements of~$A^+$.

\begin{prop}
If $S\in K\lla A \rra$ is a cyclic formal power series, then 
\begin{equation*}
\zeta_S = \prod_{\ell \in L} \, \frac{1}{1-(S,\ell) \, t^{|\ell|}} \, .
\end{equation*}
\end{prop}

\begin{proof}
Since both sides of the equation have the same constant term~$1$, it suffices to prove that they have the 
same logarithmic derivative. The logarithmic derivative of the RHS multiplied by~$t$ is equal to
\begin{equation*}
\sum_{\ell\in L} \, \frac{|\ell| \, (S,\ell) \, t^{|\ell|}}{1-(S,\ell) \, t^{|\ell|}} \, ,
\end{equation*}
which in turn is equal to
\begin{equation*}
\sum_{\ell\in L, \, k\geq 1} \, |\ell| \, (S,\ell)^k \, t^{k|\ell|} \, .
\end{equation*}
In view of~\eqref{def-g} and~\eqref{dlog} it is enough to check that for all $nÊ\geq 1$,
\begin{equation}\label{check}
\sum_{|w| = n} \, (S,w) = \sum_{\ell\in L, \, k\geq 1, \, k|\ell| = n} \, |\ell| \, (S,\ell)^k \, .
\end{equation}
Now any word~$w = u^k$ is the $k$-th power of a unique primitive word~$u$, 
which is the conjugate of a unique element~$\ell\in L$.
Moreover, $w$ has exactly $|\ell|$~conjugates and since~$S$ is cyclic, we have
\begin{equation*}
(S,w) = (S,u^k) = (S,u)^k = (S,\ell)^k .
\end{equation*}
From this Equation~\eqref{check} follows immediately.
\end{proof}

\begin{coro}\label{zeta-entier-coro}
If a cyclic formal power series~$S$ has integer coefficients, i.e., $(S,w) \in \ZZ$ for all $w\in A^*$, then 
so has $\zeta_S$.
\end{coro}

\section{Algebraic noncommutative series}\label{sec-alg}

This section is essentially a compilation of well-known results on algebraic noncommutative series.

Recall that a 
\emph{system of proper algebraic noncommutative equations} is a finite set of equations
\begin{equation*}
\xi_i = p_i \, \quad i=1, \ldots, n \, ,
\end{equation*}
where $\xi_1, \ldots, \xi_n$ are noncommutative variables and $p_1, \ldots, p_n$
are elements of $K\langle \xi_1, \ldots, \xi_n, A\rangle$,
where $A$ is some alphabet. 
We assume that each~$p_i$ has no constant term and contains no monomial~$\xi_j$.
One can show that such a system has a unique solution $(S_1, \ldots, S_n)$, i.e., 
there exists a unique $n$-tuple $(S_1, \ldots, S_n) \in K\lla A \rra ^n$ such that
$S_i = p_i(S_1, \ldots, S_n, A)$ for all $i=1, \ldots, n$
and each~$S_i$ has no constant term
(see~\cite{Sch}, \cite[Th.~IV.1.1]{SS}, or \cite[Prop.\,6.6.3]{St}). 

If a formal power series $S \in K\lla A \rra$ differs by a constant from such a formal power series~$S_i$, we say that
$S$ is \emph{algebraic}.

\begin{exa}\label{exa-Lukas}
Consider the proper algebraic noncommutative equation
\begin{equation*}
\xi = a\xi^2 + b \, .
\end{equation*}
(Here $A = \{a,b\}$.) Its solution is of the form
\begin{equation*}
S = b + abb + aabbb + ababb + \cdots  .
\end{equation*}
One can show (see~\cite{Be}) that $S$ is the characteristic series of \L ukasiewicz's language, namely of the set of words
$w\in \{a,b\}^*$ such that $|w|_b = |w|_a +1$ and $|u|_a \geq |u|_b$ for all proper prefixes~$u$ of~$w$.
\end{exa}

Recall also that $S \in K\lla A \rra$ is \emph{rational} if it belongs to the smallest subalgebra of $K\lla A \rra$ 
containing~$K\langle A \rangle$ and closed under inversion.
By a theorem of Sch\"utzenberger (see\,\cite[Th.\,I.7.1]{BR2}), a formal power series $S \in K\lla A \rra$ is rational 
if and only if it is \emph{recognizable}, i.e., 
there exist an integer $n \geq 1$, 
a representation~$\mu$ of the free monoid~$A^*$ by matrices with entries in~$K$,
a row-matrix $\alpha$ and a column-matrix $\beta$ such that for all $w\in A^*$,
\begin{equation*}
(S,w) = \alpha \mu (w) \beta\, .
\end{equation*}

We now record two well-known theorems. 

\begin{theorem}\label{CS-thm}
(1) If $S\in K\lla A \rra$ is algebraic, then its generating series $g_S \in K[[t]]$ is algebraic in the usual sense.

(2) The set of algebraic power series is a subring of~$K\lla A \rra$.

(3) A rational power series is algebraic.

(4) The Hadamard product of a rational power series and an algebraic 
power series is algebraic. 

(5) Let $A= \{a_1, \ldots, a_n, a_1^{-1}, \ldots, a_n^{-1}\}$ 
and $L$ be the language consisting of all words on the alphabet~$A$ whose image 
in the free group on $a_1, \ldots, a_n$ is the identity element.
Then the characteristic series of~$L$ is algebraic.
\end{theorem}

Items (1)--(4) of the previous theorem are due to Sch\"utzenberger~\cite{Sch}, 
Item~(5) to Chomsky and Sch\"utzenberger~\cite{CS} (see~\cite[Example\,6.6.8]{St}).

The second theorem is a criterion due to Jacob~\cite{Ja}.

\begin{theorem}\label{Ja-thm}
A formal power series $S\in K\lla A \rra$ is algebraic if and only if  
there exist a free group~$F$, a representation~$\mu$ of the free monoid~$A^*$
by matrices with entries in~$KF$, indices~$i,j$ and $\gamma \in F$ such that for all $w\in A^*$,
\begin{equation*}
(S,w) = \bigl( (\mu w)_{i,j}, \gamma \bigr).
\end{equation*}
\end{theorem}

The following is an immediate consequence of Theorem~\ref{Ja-thm}.

\begin{coro}\label{functoriality}
If $S\in K\lla A \rra$ is an algebraic power series and $\varphi: B^* \to A^*$ is a homomorphism
of finitely generated free monoids, then the power series
\begin{equation*}
\sum_{w\in B^*} \, \left(S, \varphi(w) \right)\, w \in K\lla B \rra
\end{equation*}
is algebraic.
\end{coro}

As a consequence of Theorem~\ref{CS-thm}\,(5) and of Corollary~\ref{functoriality}, we obtain the following.

\begin{coro}\label{coro-alg}
Let $f: A^* \to F$ be a homomorphism from~$A^*$ to a free group~$F$. 
Then the characteristic series of~$f^{-1}(1)\in K\lla A \rra$ is algebraic.
\end{coro}

\section{Proof of Theorem~\ref{main}}\label{proof-main}

Let $M$ be a $d\times d$-matrix. 
As observed in the introduction, it is enough to establish Theorem~\ref{main} when all the entries of~$M$
belong to $\ZZ F$.

We first reformulate the formal power series~$g_M$ and~$P_M$
of~\eqref{PM} as the generating series and the zeta function of a noncommutative formal power series,
respectively.

Let $A$ be the alphabet whose elements are triples $[g,i,j]$, where $i,j$ are integers such that
$1\leq i,j \leq d$ and $g \in F$ appears in the $(i,j)$-entry~$M_{i,j}$ of~$M$,
i.e., $(M_{i,j},g) \neq 0$.
We define the noncommutative formal power series $S_M \in K\lla A \rra$ as follows: 
for $w = [g_1,i_1,j_1] \cdots [g_n,i_n,j_n] \in A^+$, the scalar $(S_M,w)$ vanishes unless we have
\begin{itemize}
\item[(a)]
$j_n = i_1$ and $j_k = i_{k+1}$ for all $k=1, \ldots, n-1$,

\item[(b)]
$g_1\cdots g_n = 1$ in the group~$F$,
\end{itemize}
in which case $(S_M,w)$ is given by 
\begin{equation*}
(S_M,w) = (M_{i_1, j_1}, g_1) \cdots (M_{i_n, j_n}, g_n) \in K \, .
\end{equation*}
By convention, $(S_M,1) = d$.

\begin{prop}\label{identification}
The generating series and the zeta function of~$S_M$ are related to
the formal power series $g_M$ and $P_M$ of\,\eqref{PM} by
\begin{equation*}
g_{S_M} - d = g_M \quad\text{and}\quad \zeta_{S_M} = P_M \, .
\end{equation*}
\end{prop}

\begin{proof}
For $n\geq 1$ we have
\begin{eqnarray*}
\Tr(M^n) 
& = & \sum\, M_{i_1, j_1}  \cdots M_{i_n, j_n} \\
& = & \sum\, (M_{i_1, j_1}, g_1)  \cdots (M_{i_n, j_n},g_n)\, g_1 \cdots g_n \, ,
\end{eqnarray*}
where the sum runs over all indices $i_1, j_1, \ldots, i_n, j_n$ satisfying Condition\,(a) above
and over all $g_1, \ldots, g_n \in F$. 
Then
\begin{equation*}
a_n(M) = (\Tr(M^n),1)
= \sum\, (M_{i_1, j_1}, g_1)  \cdots (M_{i_n, j_n},g_n) \, ,
\end{equation*}
where Conditions\,(a) and\,(b) are satisfied. Hence,
\begin{equation*}
a_n(M) = \sum_{w\in A^* , \, |w| = n}\, (S,w) \, ,
\end{equation*}
which proves the proposition in view of\,\eqref{PM},\,\eqref{def-g} and\,\eqref{def-z}.
\end{proof}

We next establish that $S_M$ is both cyclic in the sense of Section~\ref{sec-cyclic}
and algebraic in the sense of Section~\ref{sec-alg}.

\begin{prop}\label{cyclic-prop}
The noncommutative formal power series~$S_M$ is cyclic.
\end{prop}

\begin{proof}
(i) Conditions\,(a) and\,(b) above are clearly preserved under cyclic permutations. Hence we also have
\begin{equation*}
(S_M,w) = (M_{i_2, j_2}, g_2) \cdots (M_{i_n, j_n}, g_n) \, (M_{i_1, j_1}, g_1)
\end{equation*}
when $w = [g_1,i_1,j_1] \cdots [g_n,i_n,j_n]$ such that Conditions\,(a) and\,(b) are satisfied.
It follows that $(S,uv) = S(vu)$ for all $u,v \in A^*$.

(ii) If $w$ satisfies Conditions\,(a) and\,(b), so does $w^r$ for $r\geq 2$. 
Conversely, if $w^r$ ($r\geq 2$) satisfies Condition\,(a), then since 
\begin{equation*}
w^r = [g_1,i_1,j_1] \cdots [g_n,i_n,j_n]\, [g_1,i_1,j_1] \cdots  
\end{equation*}
we must have $j_n = i_1$ and $j_k = i_{k+1}$ for all $k=1, \ldots, n-1$,
and so $w$ satisfies Condition\,(a).

If $w^r$ ($r\geq 2$) satisfies Condition\,(b), i.e., $(g_1\cdots g_n)^r = 1$, then $g_1\cdots g_n = 1$ since $F$ is torsion-free.
Hence $w$ satisfies Condition\,(b).
It follows that
$(S,w^r) = \left((M_{i_1, j_1}, g_1) \cdots (M_{i_n, j_n}, g_n)\right)^r = (S,w)^r$.
\end{proof}

\begin{prop}\label{alg-prop}
The noncommutative formal power series~$S_M$ is algebraic.
\end{prop}

\begin{proof}
We write $S_M$ as the Hadamard product of three noncommutative formal power series $S_1, S_2, S_3$.

The series~$S_1 \in K\lla A \rra$ is defined
for $w = [g_1,i_1,j_1] \cdots [g_n,i_n,j_n] \in A^+$ by
\begin{equation*}
(S_1,w) = (M_{i_1, j_1}, g_1) \cdots (M_{i_n, j_n}, g_n) 
\end{equation*}
and by $(S_1,1) = 1$.
This is a recognizable, hence rational, series with one-dimensional representation $A^* \to K$ given by
$[g,i,j] \mapsto (M_{i,j}, g)$.

Next consider the representation~$\mu$ of the free monoid~$A^*$ defined by
\begin{equation*}
\mu([g,i,j]) = E_{i,j}\, ,
\end{equation*}
where $E_{i,j}$ denotes as usual the $d \times d$-matrix with all entries vanishing, 
except the $(i,j)$-entry which is equal to~$1$.
Set
\begin{equation*}
S_2 = \sum_{w\in A^*}\, \Tr\left( (\mu w) \right) \, w  \in K\lla A \rra \, .
\end{equation*}
The power series~$S_2$ is recognizable, hence rational. 
Let us describe~$S_2$ more explicitly.
For $w= 1$, $\mu(w)$ is the identity $d \times d$-matrix; hence $(S_2,1) = d$.
For $w = [g_1,i_1,j_1] \cdots [g_n,i_n,j_n] \in A^+$, we have 
\begin{equation*}
\Tr\left( (\mu w) \right) = \Tr(E_{i_1,j_1} \cdots E_{i_n,j_n}) \, .
\end{equation*}
It follows that $\Tr\left( (\mu w) \right) \neq 0$ if and only if $\Tr(E_{i_1,j_1} \cdots E_{i_n,j_n}) \neq 0$,
which is equivalent to $j_n = i_1$ and $j_k = i_{k+1}$ for all $k=1, \ldots, n-1$, in which case
$\Tr\left( (\mu w) \right) = 1$.
Thus, 
\begin{equation*}
S_2 = d + \sum_{n\geq 1} \, \sum\, [g_1,i_1,i_2] \, [g_2,i_2,i_3] \cdots [g_n,i_n,i_1] \, ,
\end{equation*}
where the second sum runs over all elements $g_1, \ldots, g_n \in F$ and all indices $i_1, \ldots, i_n$.

Finally consider the homomorphism $f: A^* \to F$ sending $[g,i,j]$ to~$g$. 
Then by Corollary~\ref{coro-alg} the characteristic series~$S_3 \in K\lla A \rra$ of~$f^{-1}(1)$ is algebraic.

It is now clear that $S_M$ is the Hadamard product of $S_1$, $S_2$, and~$S_3$:
\begin{equation*}
S_M = S_1 \odot S_2 \odot S_3 \, .
\end{equation*}
Since by~\cite[Th.~I.5.5]{BR2} the Hadamard product of two rational series is rational,
$S_1 \odot S_2$ is rational as well. 
It then follows from Theorem~\ref{CS-thm}\,(4) and the algebraicity of~$S_3$
that $S_M = S_1 \odot S_2 \odot S_3$ is algebraic.
\end{proof}

Since $M$ has entries in~$\ZZ F$, the power series $g_{S_M} = g_M + d$ belongs to~$\ZZ[[t]]$.
It follows by Corollary~\ref{zeta-entier-coro} and Proposition~\ref{cyclic-prop}
that the power series~$P_M = \zeta_{S_M}$ has integer coefficients as well.
Moreover, by Theorem~\ref{CS-thm}\,(1) and Proposition~\ref{alg-prop},
\begin{equation*}
t \, \frac{d \log(P_M)}{dt}  = g_M
\end{equation*}
is algebraic. 

To complete the proof of Theorem\,\ref{main}, it suffices to apply the following algebraicity theorem.

\begin{theorem}\label{thmCC}
If $f \in \ZZ[[t]]$ is a formal power series with integer coefficients such that $t \, d \log f/dt$ is algebraic, then
$f$ is algebraic.
\end{theorem}

Note that the integrality condition for~$f$ is essential: for the transcendental formal power series $f= \exp(t)$,
we have $t \, d \log f/dt = t$, which is even rational.

\begin{proof}
This result follows from cases of the Grothendieck--Katz conjecture proved in\,\cite{An2} and in\,\cite{Bo}. 
The conjecture states that if $Y' = AY$ is a linear system of differential equations
with $A \in M_d(\QQ(t))$, then far from the poles of~$A$ it has a basis of solutions that are algebraic over~$\QQ(t)$
if and only if for almost all prime numbers~$p$ the reduction mod\,$p$ of the system has a basis of
solutions that are algebraic over~$\FF_p(t)$.

Let us now sketch a proof of the theorem (see also Exercise~5 of\,\cite[p.\,160]{An1}). 
Set $g = tf'/f$ and consider the system $y' = (g/t)\, y$;
it defines a differential form~$\omega$ on an open set~$S$ 
of the smooth projective complete curve~$\overline S$ associated to~$g$.
We now follow\,\cite[\S~6.3]{An2}, which is inspired from\,\cite{CC}. 
First, extend $\omega$ to a section (still denoted~$\omega$) of~$\Omega^1_{\overline S}(-D)$, 
where $D$ is the divisor of poles of~$\omega$. For any $n\geq 2$,
we have a differential form $\sum_{i=1}^n\, p^*_i(\omega)$ on~$S^n$,
where $p_i : S^n \to S$ is the $i$-th canonical projection; 
this form goes down to the symmetric power~$S^{(n)}$.
Now let $J$ be the generalized Jacobian of~$S$ which parametrizes the invertible fibre bundles 
over~$\overline S$ that are rigidified over~$D$.
There is a morphism $\varphi: S \to J$ and a unique invariant differential form~$\omega_J$ on~$J$ such that 
$\omega = \varphi^*(\omega_J)$. For any~$n\geq 2$, $\varphi$ induces a morphism $\varphi^{(n)}: S^{(n)} \to J$ 
such that $(\varphi^{(n)})^*(\omega_J) = \sum_{i=1}^n\, p^*_i(\omega)$.
For $n$ large enough, $\varphi^{(n)}$ is dominant
and if~$\omega_J$ is exact, then so is~$\omega$. 
To prove that $\omega_J$ is exact, we note that $J$, being a scheme of commutative groups, is uniformized by~$\CC^n$.
We can now apply Theorem~5.4.3 of \emph{loc.\ cit}., 
whose hypotheses are satisfied because the solution~$f$ of the system has integer coefficients. 

Alternatively, one can use a special case of a generalized conjecture of Grothendieck-Katz proved by Bost, namely
Corollary\,2.8 in\,\cite[Sect.\,2.4]{Bo}: the vanishing of the $p$-curvatures in Condition\,(i) follows by a theorem of Cartier
from the fact that the system has a solution in $\FF_p(t)$, namely the reduction mod~$p$ of~$f$ for all prime numbers~$p$
for which such a reduction of the system exists (see Exercise\,3 of \cite[p.~84]{An1} or Theorem\,5.1 of~\cite{Ka}); 
Condition\,(ii) is satisfied since $\CC^n$ satisfies the Liouville property.
\end{proof}

A nice overview of such algebraicity results is given in Chambert-Loir's Bourbaki report\,\cite{Ch}; see especially Theorem~2.6
and the following lines.

\section{Examples}\label{sec-exa}

Kontsevich\,\cite{Ko} computed $P_{\omega}$ when ${\omega} = X_1 + X_1^{-1} + \cdots + X_n + X_n^{-1}$
considered as a $1 \times 1$-matrix, obtaining
\begin{equation}\label{eq-Pa}
P_{\omega} = \frac{2^n}{(2n-1)^{n-1}} \cdot 
\frac{\bigl( n-1 + n \bigl( 1-4(2n-1)t^2 \bigr)^{1/2}\bigr)^{n-1}}{\bigl( 1 + \bigl(1-4(2n-1)t^2 \bigr)^{1/2} \bigr)^n} ,
\end{equation}
which shows that $P_{\omega}$ belongs to a quadratic extension of~$\QQ(t)$. 

We now present similar results for the zeta functions of two matrices, 
the first one of order~$2$, the second one of order~$d\geq 3$.

\subsection{Computing $P_M$ for a $2\times 2$-matrix}\label{exa-M2}

Consider the  following matrix with entries in the ring $\ZZ\langle a, a^{-1}, b, b^{-1}, d, d^{-1} \rangle$,
where $a$, $b$, $d$ are noncommuting variables:
\begin{equation}\label{mat-2}
M =
\begin{pmatrix}
a + a^{-1} & b \\
b^{-1} & d + d^{-1}
\end{pmatrix}
.
\end{equation}

\begin{prop}
We have
\begin{equation}\label{eq-gM}
g_M = 3 \, \frac{(1-8t^2)^{1/2} -1+6t^2}{1-9t^2}
\end{equation}
and
\begin{equation}\label{eq-PM}
P_M = \frac{(1-8t^2)^{3/2} - 1 +12t^2 -24 t^4}{32t^6}\, .
\end{equation}
\end{prop}

Expanding~$P_M$ as a formal power series, we obtain 
\begin{equation*}
P_M  = 1 + \sum_{n\geq 1}\, \frac{3 \cdot 2^n}{(n+2)(n+3)} \, \binom{2n+2}{n+1} \,  t^{2n}  \, .
\end{equation*}

\begin{proof}
View the matrix~$M$ under the form of the graph of Figure\,1, with two vertices~$1,2$ and
six labeled oriented edges.
We identify paths in this graph and words on the alphabet $A = \{a, a^{-1}, b, b^{-1}, d, d^{-1}\}$.
Let $B$ denote the set of nonempty words on~$A$ which become trivial
in the corresponding free group on~$a,b,d$ and whose corresponding path is a closed path.
Then the integer $a_n(M)$ is the number of words in~$B$ of length~$n$.
We have $\eps(B) = g_M$, where $\eps : K\lla A \rra \to K[[t]]$ is the algebra map
defined in Section~\ref{general}.

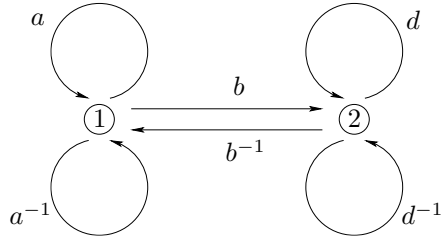
\begin{figure}[ht!]
\labellist
\small\hair 2pt
\pinlabel $1$ [r] at 52 100
\pinlabel $2$ [r] at 268 100
\pinlabel $d$ [r] at 318 185
\pinlabel $d^{-1}$ [r] at 338 19
\pinlabel $a$ [r] at 0 185
\pinlabel $a^{-1}$ [r] at 6 19
\pinlabel $b$ [r] at 169 128
\pinlabel $b^{-1}$ [r] at 187 73
\endlabellist
\centering
\includegraphics[scale=0.45]{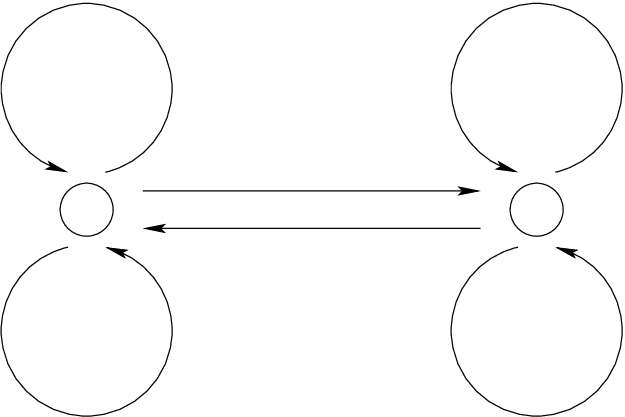}
\caption{\emph{A graph representing~$M$}}
\label{graph1}
\end{figure}

We define $B_i$ ($i= 1,2$) as the set of paths in~$B$ starting from and ending at the vertex~$i$; 
we have $B = B_1 + B_2$.
Each set~$B_i$ is a free subsemigroup of~$A^*$, freely generated by the set~$C_i$ of closed paths not passing through~$i$
(except at their ends).
The sets~$C_i$ do not contain the empty word.
We have
\begin{equation*}
B_i = C_i^+ = \sum_{n\geq 1}\, C_i^n  \qquad (i=1,2)
\end{equation*}

Given a letter~$x$, we denote by~$C_i(x)$ the set of closed paths in~$C_i$ starting with~$x$.
Any word of~$C_i(x)$ is of the form~$xwx^{-1}$, where $w \in B_j$ when $i \overset{x}{\longrightarrow} j$;
such~$w$ does not start with~$x^{-1}$. 
Identifying a language with its characteristic series and 
using the standard notation $L^* = 1 + \sum_{n\geq 1} \, L^n$ for any language~$L$,
we obtain the following two equations:
\begin{equation}\label{eq-a}
C_1(a) = a ( C_1(a) + C_1(b))^* a^{-1}
\end{equation}
and 
\begin{equation}\label{eq-b}
C_1(b) = b( C_2(d) + C_2(d^{-1}))^* b^{-1} \, .
\end{equation}

Applying the algebra map~$\eps$ and taking into account the symmetries of the graph, 
we see that the four noncommutative formal power series $C_1(a)$, $C_1(a^{-1})$, $C_2(d)$, $C_2(d^{-1})$
are sent to the same formal power series~$u \in \ZZ[[t]]$, 
while $C_1(b)$, $C_2(b^{-1})$ are sent to the same formal power series~$v$.
It follows from\,\eqref{eq-a} and\,\eqref{eq-b} that $u$ and $v$ satisfy the equations
\begin{equation}\label{eq-uv}
u = t^2(u+v)^* = \frac{t^2}{1 - u - v}
\quad\text{and} \quad
v = t^2(2u)^* = \frac{t^2}{1 - 2u} \, ,
\end{equation}
from which we deduce
\begin{equation*}
t^2 = u(1 - u - v) =  v(1- 2u)\, .
\end{equation*}
The second equality is equivalent to~$(u-v)(u-1) = 0$. Since $C_1(a)$ does not contain the empty word,
the constant term of~$u$ vanishes, hence $u-1 \neq 0$. Therefore, $u=v$.

Since $C_1 = C_1(a) + C_1(a^{-1}) + C_1(b)$ and $C_2 = C_2(d) + C_2(d^{-1}) + C_2(b^{-1})$,
we have 
$\eps(C_1) = \eps(C_2) = 2u + v = 3u$.
Therefore, $\eps(B_1) = \eps (B_2) = 3u/(1-3u)$ and
\begin{equation}\label{eq-gu}
g_M = \eps(B) = \frac{6u}{1-3u} \, .
\end{equation}
Let us now compute~$u$ using\,\eqref{eq-uv} and the equality $u=v$. 
The formal power series~$u$ satisfies the quadratic equation $2u^2 - u + t^2 = 0$.
Since $u$ has zero constant term, we obtain
\begin{equation*}
u = \frac{1 - (1 - 8t^2)^{1/2}}{4} \, .
\end{equation*}
From this and~\eqref{eq-gu}, we obtain the desired form for~$g_M$.

Let $P(t)$ be the right-hand side in Equation\,\eqref{eq-PM}. To prove $P_M = P(t)$, 
we checked that $ t P'(t)/P(t) = g_M$ and the constant term of~$P(t)$ is~$1$. 
\end{proof}

\begin{rem}\label{rem-PM}
We found Formula\,\eqref{eq-PM} for~$P(t)$ as follows. 
We first computed the lowest coefficients of~$g_M$ up to degree~$10$:
\begin{equation*}
g_M = 6\, (t^2 + 5 t^4 + 29 t^6 + 181 t^8 + 1181 t^{10})+ O(t^{12}) \, .
\end{equation*}
From this it was not difficult to find that
\begin{equation}\label{seq-PM}
P_M = 1 + 3 \, t^2 + 12 \, t^4 + 56 \, t^6 + 288 \, t^8 + 1584 \, t^{10} + O(t^{12}) \, .
\end{equation}
Up to a shift, the sequence\,\eqref{seq-PM} of nonzero coefficients of~$P_M$ 
is the same as the sequence of numbers of ``new'' intervals in a Tamari lattice computed by Chapoton in~\cite[Sect.~9]{Cp}.
(We learnt this from\,\cite{OEIS} where this sequence is listed as~A000257.)
Chapoton gave an explicit formula for the generating function~$\nu$ of these ``new'' intervals (see Eq.~(73) in \emph{loc. cit.}).
Rescaling~$\nu$, we found that $P(t) = (\nu(t^2) - t^4)/t^6$ has up to degree~$10$
the same expansion as~\,\eqref{seq-PM}. It then sufficed to check that $tP'(t)/P(t) = g_M$.

By\,\cite{OEIS} the integers in the sequence~A000257 also count the number of planar rooted bicubic maps 
with $2n$ vertices (see\,\cite[p.\,269]{Tu}).
Planar maps also come up in the combinatorial interpretation of\,\eqref{eq-Pa} 
given in\,\cite[Sect.~5]{RR} for $n=2$.
 
Note that the sequence of nonzero coefficients of $g_M/6$ is listed as A194723 in\,\cite{OEIS}.
\end{rem}

\subsection{A similar $d\times d$-matrix}\label{exa-Md}

Fix an integer $d\geq 3$ and let $M$ be the $d\times d$-matrix
with entries~$M_{i,j}$ defined by
\begin{equation*}
M_{i,i} = a_i + a_i^{-1}
\qquad\text{and}\qquad 
M_{i,j} = 
\begin{cases} 
\, b_{ij} & \text{if $i < j$\! ,}
\\
\, b_{ji}^{-1} &\text{if $j < i$\! ,}
\end{cases}
\end{equation*}
where $a_1, \ldots, a_d$, $b_{ij}$ ($1 \leq i < j \leq d$) are noncommuting variables.
This matrix is a straightforward generalization of\,\eqref{mat-2}.

Proceeding as above, we obtain two formal power series $u$ and $v$ satisfying 
the following equations similar to\,\eqref{eq-uv}:
\begin{equation*}
u = t^2(u+(d-1)v)^* = \frac{t^2}{1 - u - (d-1)v}
\end{equation*}
and
\begin{equation*}
v = t^2(2u + (d-2)v)^* = \frac{t^2}{1 - 2u - (d-2)v} \, ,
\end{equation*}
We deduce the equality $u=v$ and the quadratic equation $u(1-du) = t^2$.
We finally have
\begin{equation*}
g_M = \frac{d(d+1) u}{1 - (d+1)u} \, ,
\end{equation*}
which leads to
\begin{equation*}
g_M = \frac{d(d+1)}{2} \, \frac{(1-4d t^2)^{1/2} - 1 + 2 (d+1) t^2}{1 - (d+1)^2 t^2} .
\end{equation*}

Its expansion as a formal power series is the following:
\begin{eqnarray*}\label{seq-gM}
g_M & = & d(d+1) \, \bigl\{t^2 + (2d+1) \, t^4 + (5 d^2+4 d+1) \, t^6 \\
&&  {} + (14 d^3+14 d^2+6 d+1) \, t^8 \\
&& {} + (42 d^4+48 d^3+27 d^2+8 d+1) \, t^{10}Ê\bigr\} + O(t^{12}) \, .
\end{eqnarray*}
When $d = 2, 3, 4$, the sequence of nonzero coefficients of~$g_M/d(d+1)$ 
is listed respectively as A194723, A194724, A194725 in\,\cite{OEIS} 
(it is also the $d$-th column in Sequence A183134). 
These sequences count the $d$-ary words either empty or beginning with the first letter of the alphabet, 
that can be built by inserting $n$~doublets into the initially empty word.

We were not able to find a closed formula for~$P_M$ analogous to\,\eqref{eq-PM}. 
Using Maple, we found that, for instance up to degree~$10$, the expansion of~$P_M$ is
\begin{eqnarray*}
&& 1 + \frac{d(d+1)}{2} \, t^2 + \frac{d(d+1)(d^2+5 d+2)}{8} \, t^4 \\
&&  {} + \frac{d(d+1)(d^4+14 d^3+59 d^2+38 d+8)}{48} \, t^6 \\
&&  {} + \frac{d(d+1)(d^6+27 d^5+271 d^4+1105 d^3+904 d^2+332 d +48)}{384} \, t^8  \, .
\end{eqnarray*}

\section*{Acknowledgement}

We are most grateful to Yves Andr\'e, Jean-Beno\^\i t Bost and Carlo Gasbarri for their help in 
the proof of Theorem\,\ref{thmCC}.
We are also indebted 
to Fran\c cois Ber\-ge\-ron and Pierre Guillot for assisting us with computer computations
in the process detailed in Remark\,\ref{rem-PM},
to Fr\'ed\'eric Chapoton for his comments on Section\,\ref{exa-Md},
and to an anonymous referee for having spotted slight inaccuracies.
Thanks also to Stavros Garoufalidis for having pointed out References~\cite{GB, Sa}.

The first-named author was partially funded by LIRCO 
(Laboratoire International Franco-Qu\'eb\'ecois de Recherche en Combinatoire)
and UQAM (Universit\'e du Qu\'ebec \`a Montr\'eal).
The second-named author is supported by NSERC (Canada).

\end{document}